\documentclass[11pt,a4paper]{amsart}
\usepackage{mathrsfs}
\usepackage[T1]{fontenc}
\usepackage[utf8]{inputenc}
\usepackage{lmodern,microtype,mathtools,amssymb}
\usepackage[a4paper,margin=26mm]{geometry}
\usepackage{enumitem}
\usepackage[colorlinks=true,linkcolor=blue,citecolor=blue,urlcolor=blue]{hyperref}
\setlist[enumerate]{label=\textup{(\roman*)},leftmargin=2em}
\newcommand{\N}{\mathcal N(e)}
\newcommand{\Pow}{\mathcal P}
\DeclareMathOperator{\core}{core}
\DeclareMathOperator{\Stab}{Stab}
\DeclareMathOperator{\Sym}{Sym}
\DeclareMathOperator{\cf}{cf}
\DeclareMathOperator{\ib}{ib}
\theoremstyle{plain}
\newtheorem{theorem}{Theorem}[section]
\newtheorem{lemma}[theorem]{Lemma}
\newtheorem{definition}[theorem]{Definition}

\newtheorem{corollary}[theorem]{Corollary}
\theoremstyle{remark}

\numberwithin{equation}{section}
\title{The answer about Itzkowitz's Problems on FSIN groups}
\author{Dekui Peng}
\address{Institute of Mathematics, Nanjing Normal University,
Nanjing 210024, China}
\email{pengdk10@lzu.edu.cn}
\author{Li-Hong Xie}
\address{School of Mathematics and Computational Science, Wuyi University,
Jiangmen, Guangdong 529000, P.R. China}
\email{yunli198282@126.com}
\author{Jiang Yang}
\address{School of Mathematical Sciences, Guangxi Minzu University,
Nanning 530006, P.R. China}
\email{yangjjiangdy@126.com}
\thanks{Jiang Yang is the corresponding author.}

\subjclass[2020]{Primary 22A05; Secondary 54E15, 54A25, 03E10, 03E35}
\keywords{Functionally balanced group, SIN group, Itzkowitz problem,
  neutral subset, uniform continuity}

\begin{document}
\begin{abstract}
A topological group is functionally balanced if every bounded real-valued
left uniformly continuous function is right uniformly continuous.
We prove that every Hausdorff functionally balanced group has coinciding
left and right uniformities, giving an affirmative answer to the
Itzkowitz problem. Consequently, the classes of SIN, SFSIN and FSIN
groups coincide.
\end{abstract}
\maketitle

\section{Introduction}

All topological groups in this paper are Hausdorff. For a group $G$ with
identity $e$, let $\N$ denote its identity-neighbourhood filter.
The left and right uniformities are generated, respectively, by
\[
 \{(x,y):x^{-1}y\in U\}
 \quad\text{and}\quad
 \{(x,y):xy^{-1}\in U\},\qquad U\in\N.
\]
A group is \emph{balanced}, or \emph{SIN}, if these uniformities coincide.
Equivalently, every identity neighbourhood contains an identity
neighbourhood invariant under all inner automorphisms.
A group is \emph{functionally balanced}, or \emph{FSIN}, if every bounded
real-valued left uniformly continuous function is right uniformly
continuous. If the same holds without the boundedness restriction,
the group is \emph{strongly functionally balanced}, or \emph{SFSIN}.
Thus
\[
 [\mathrm{SIN}]\subseteq[\mathrm{SFSIN}]\subseteq[\mathrm{FSIN}].
\]
The comparison of the two uniformities through real-valued uniformly
continuous functions already appears in the work of Comfort and
Ross~\cite[p.~484]{CR66}. The Itzkowitz problem asks whether functional
balance implies balance; see Bouziad and Troallic~\cite{BT07} for its
bounded and unbounded versions.

The problem has affirmative answers under several additional hypotheses.
Milnes~\cite{Milnes90} treated locally compact groups, and
Protasov~\cite{Protasov91} proved the result for almost metrizable groups.
For locally connected groups, see Megrelishvili, Nickolas and
Pestov~\cite[Main Theorem]{MNP97}, together with the bounded FSIN
formulation in~\cite[Corollary 3.11]{BT04}.
Algebraic hypotheses also give positive results: Bareche and
Bouziad~\cite[Corollary 3.7]{BB10} proved that every FSIN group of
finite exponent is SIN.

A subset $A$ of $G$ is \emph{right neutral} if, for every $V\in\N$,
there is $U\in\N$ such that $AU\subseteq VA$.
It is \emph{left neutral} if, for every $V\in\N$,
there is $U\in\N$ such that $UA\subseteq AV$.
Inversion interchanges these two properties.
We use the following characterization.

\begin{theorem}[{\cite[Theorem 2.2]{BT04}}]\label{thm:neutral}
For a topological group $G$, the following conditions are equivalent:
\begin{enumerate}
\item $G$ is FSIN;
\item every subset of $G$ is right neutral;
\item every subset of $G$ is left neutral.
\end{enumerate}
\end{theorem}

The equivalence of the first two conditions is the Protasov--Saryev
criterion in the form given in~\cite{BT04}.
Applying right neutrality to $A^{-1}$ and taking inverses gives the
third condition, and conversely.
Whenever neutrality supplies a neighbourhood $U$, we may choose $U$
to be open, symmetric, and contained in any previously prescribed
identity neighbourhood. We make such choices directly below.

Let $\kappa$ be an infinite cardinal. A subset $A$ of $G$ is
\emph{left $\kappa$-precompact} if, for every $U\in\N$, there is
$F\subseteq G$ such that
\[
 |F|<\kappa\qquad\text{and}\qquad A\subseteq FU.
\]
Right $\kappa$-precompactness is defined using $A\subseteq UF$.
For $A=G$, the two conditions are equivalent by inversion;
we then simply say that $G$ is \emph{$\kappa$-precompact}.
In particular, $\aleph_0$-precompactness is ordinary precompactness.
A group $G$ is \emph{$\kappa$-narrow} if every $U\in\N$ admits
$G=FU$ with $F\subseteq G$ and $|F|\le\kappa$. Thus
\begin{equation}\label{eq:narrow}
 G\text{ is $\kappa^+$-precompact}
 \quad\Longleftrightarrow\quad
 G\text{ is $\kappa$-narrow}.
\end{equation}
This agrees with the usual narrowness convention; see
\cite[Section~5.1]{AT08}. The strict inequality in the definition of
$\kappa$-precompactness will be used throughout. A group is \emph{locally precompact} if it
has a precompact identity neighbourhood.

\begin{theorem}[{\cite[Corollary 3.7]{BT04}}]\label{thm:local}
Every locally precompact FSIN group is SIN.
\end{theorem}

We prove the following theorem.

\begin{theorem}\label{thm:main}
Every Hausdorff FSIN topological group is SIN.
\end{theorem}

In Section~\ref{sec:core}, we prove that open subgroups of an FSIN group
have open normal cores and establish a covering lemma for generated
subgroups. Section~\ref{sec:permutations} constructs local permutations
on separated subsets and families of subsets that distinguish them.
We then compare packing sizes with covering cardinals to complete the
proof. The covering lemma handles arbitrary cardinal cofinalities.

Unless otherwise stated, the topological group $G$ denotes a fixed Hausdorff FSIN group and $\N$ denotes the family of all open neighborhoods of the identity element $e$ in $G$ in
Sections~\ref{sec:permutations}. We write $\Pow(X)$ for the power set of $X$ and
$\core_G(H)=\bigcap_{g\in G}gHg^{-1}$ for the normal core of a subgroup
$H$ of $G$. Let $\Sym(X)$ denote the permutation group on the set $X$. All cardinal arguments are in ZFC.

\section{Definitions and preliminary results}
\label{sec:core}

We first show that every open subgroup of $G$ has an open normal core.
If a group $\Gamma$ acts on a set $X$ and $A\subseteq X$, we write
$\Stab_\Gamma(A)=\{\gamma\in\Gamma:\gamma A=A\}$.
For a faithful action $\Gamma\leq\Sym(X)$, a colouring $\tau:X\to I$
is \emph{distinguishing} if the identity is the only element $\gamma\in\Gamma$
such that $\tau(\gamma x)=\tau(x)$ for all $x\in X$.
We denote by $\mathrm D(\Gamma,X)$ the least cardinality of the set of
colours used by a distinguishing colouring.
The induced action on $\Pow(X)$ by a group $\Gamma$ acting on a set $X$ is given by $\gamma A=\{\gamma x:x\in A\}$ for each $A\subseteq X$,
and $\Pow(X)/\Gamma$ denotes its set of orbits.

\begin{lemma}\label{lem:stabilizer}
Let $H$ be an open subgroup of $G$. For every subset $A$ of $G/H$,
the subgroup $\Stab_G(A)$ is open in $G$.
\end{lemma}

\begin{proof}
Put $\widetilde A=\{g\in G:gH\in A\}$. Then
$\widetilde A H=\widetilde A$.
By the left neutrality of $\widetilde A$, we can choose a symmetric open
neighbourhood $U$ of $e$ such that
\[
  U\widetilde A\subseteq\widetilde A H=\widetilde A.
\]
For every $u\in U$, applying this inclusion to $u$ and $u^{-1}$ gives
$u\widetilde A=\widetilde A$. Hence $U\subseteq\Stab_G(A)$, and the
stabilizer is open.
\end{proof}

\begin{lemma}\label{lem:fibre}
Let $H$ and $Q$ be open subgroups of $G$, and put
$N=\core_G(H)$ and $X=G/H$. Suppose that $N\leq Q\leq H$.
If there is a distinguishing colouring for the action of $G/N$ on $X$
whose set of colours injects into $\Pow(H/Q)/H$, then $N$ is open in $G$.
\end{lemma}

\begin{proof}
Let $\tau:X\to I$ be such a colouring, where $I=\tau(X)$.
For each $i\in I$, choose $A_i\subseteq H/Q$ so that distinct colours
give distinct $H$-orbits. Choose $s(x)\in G$ with $s(x)H=x$ for each
$x\in X$, and put
\[
  B=\bigcup_{x\in X}s(x)A_{\tau(x)}\subseteq G/Q.
\]
Let $g\in\Stab_G(B)$. The natural projection $G/Q\to G/H$ commutes
with the action of $G$. Comparing the parts of $gB$ and $B$ in the fibre
over $gx$, we obtain
\[
  A_{\tau(gx)}=s(gx)^{-1}gs(x)A_{\tau(x)}.
\]
Since $s(gx)^{-1}gs(x)\in H$, the choice of the sets $A_i$ gives
$\tau(gx)=\tau(x)$ for every $x\in X$. The colouring is distinguishing
for $G/N$, so $g\in N$.

Conversely, $N$ is normal in $G$ and contained in $Q$. It therefore fixes
$G/Q$ pointwise. We have proved that $\Stab_G(B)=N$.
Lemma~\ref{lem:stabilizer}, applied to $Q$ and $B$, shows that $N$ is open.
\end{proof}

\begin{lemma}\label{lem:amplification}
Let $H$ and $Q$ be open subgroups of $G$, and put $N=\core_G(H)$.
Suppose that $N\leq Q\leq H$ and that $\nu=[H:Q]$ is infinite.
If $\mathrm D(G/N,G/H)\leq\nu$, then $N$ is open in $G$.
\end{lemma}

\begin{proof}
Put $Y=H/Q$, $K=\core_H(Q)$ and $\lambda=2^\nu$.
Since $N$ is normal in $H$ and contained in $Q$, we have $N\leq K$.
The faithful action of $H/K$ on $Y$ gives
\[
  |H/K|\leq|\Sym(Y)|=2^\nu=\lambda.
\]
We consider the orbits of $H$ on $\Pow(Y)$.

Suppose first that some $A\subseteq Y$ has an orbit of cardinality
$\lambda$, and put $R=\Stab_H(A)$. Regard $A$ as a subset of the fibre
$H/Q\subseteq G/Q$. Since $\lambda>1$, the set $A$ is nonempty.
If $gA=A$, the projection to $G/H$ gives $gH=H$. Thus
$\Stab_G(A)=R$, and Lemma~\ref{lem:stabilizer} shows that $R$ is open.
The subgroup $K$ fixes $Y$ pointwise and is normal in $H$. Consequently,
\[
  N\leq K\leq\core_H(R)\leq R\leq H,
  \qquad [H:R]=\lambda,
  \qquad |H/\core_H(R)|\leq\lambda.
\]
Every orbit of $H$ on $\Pow(H/R)$ therefore has cardinality at most
$\lambda$. There are $2^\lambda$ such orbits. Indeed, if their number
were $p<2^\lambda$, then their union would have cardinality at most
$\max\{p,\lambda\}<2^\lambda$, contrary to
$|\Pow(H/R)|=2^\lambda$.
Since $\mathrm D(G/N,G/H)\leq\nu<2^\lambda$,
Lemma~\ref{lem:fibre}, with $R$ in place of $Q$, shows that $N$ is open.

Suppose now that every $H$-orbit on $\Pow(Y)$ has cardinality below
$\lambda$. Fewer than $\cf(\lambda)$ such orbits have a union of
cardinality below $\lambda$. As $|\Pow(Y)|=\lambda$, it follows that
\[
  |\Pow(Y)/H|\geq\cf(2^\nu)>\nu.
\]
The last inequality is a consequence of K\"onig's theorem; see
\cite[Corollary~5.12]{Jech03}. Indeed, suppose that
$\delta=\cf(2^\nu)\leq\nu$.
K\"onig's theorem gives $(2^\nu)^\delta>2^\nu$, whereas
\[
  (2^\nu)^\delta=2^{\nu\cdot\delta}=2^\nu,
\]
a contradiction. Lemma~\ref{lem:fibre}, applied to $Q$, again shows
that $N$ is open.
\end{proof}

\begin{theorem}\label{thm:core}
For every open subgroup $H$ of $G$, the normal core $\core_G(H)$ is open.
\end{theorem}

\begin{proof}
Put $X=G/H$ and $N=\core_G(H)$. Choose a set $T\subseteq X$ containing
exactly one point from each $H$-orbit. By Lemma~\ref{lem:stabilizer},
the subgroup
\[
  Q=H\cap\Stab_G(T)
\]
is open. Since $N$ fixes $X$ pointwise, we have $N\leq Q$.
An element of $H$ preserving $T$ fixes its unique point in each
$H$-orbit. Therefore
\[
  Q=\bigcap_{t\in T}H_t,
  \qquad H_t=\{h\in H:ht=t\}.
\]
Put
\[
  \kappa=\sup_{x\in X}|Hx|,
  \qquad \nu=[H:Q],
  \qquad d_*=\mathrm D(G/N,X).
\]
For every $t\in T$, the inclusion $Q\leq H_t$ gives
$|Ht|=[H:H_t]\leq\nu$. Since $T$ meets every $H$-orbit, we obtain
$\kappa\leq\nu$.

Give the point $H\in X$ a colour used nowhere else. Colour each other
$H$-orbit injectively, using the same set of $\kappa$ colours for all
these orbits. An element of $G$ preserving this colouring fixes $H$
and hence belongs to $H$. It then fixes every point of each $H$-orbit,
so it belongs to $N$. Thus the colouring is distinguishing for $G/N$,
and $d_*\leq\kappa+1$.

If $\nu$ is infinite, then $d_*\leq\nu$, and
Lemma~\ref{lem:amplification} applies. If $\nu=m<\omega$, then
$d_*\leq m+1$. Subsets of the $m$-element set $H/Q$ with different
cardinalities belong to different $H$-orbits. Hence
\[
  |\Pow(H/Q)/H|\geq m+1\geq d_*.
\]
In this case Lemma~\ref{lem:fibre} applies. Therefore $N$ is open in
both cases.
\end{proof}

For a topological group $K$, we denote by $\ib(K)$ the least
infinite cardinal $\kappa$ for which $K$ is $\kappa$-narrow. Thus
$\ib(K)\leq\kappa$ if and only if $K$ is
$\kappa^+$-precompact. We first record the covering properties used in
this section.

\begin{lemma}\label{lem:permanence}
Let $G$ be a topological group and $\kappa$ an infinite cardinal. Left $\kappa^+$-precompact subsets
of $G$ are closed under left and right translations,
finite products, and countable unions. If $P=P^{-1}\in\N$ is left
$\kappa^+$-precompact, then $\langle P\rangle$ is $\kappa$-narrow.
\end{lemma}

\begin{proof}
Left translations change only the covering centres. For a right
translate $Ax$ and a target $U\in\N$, cover $A$ by at most $\kappa$
left translates of $xUx^{-1}$ and multiply on the right by $x$.

Let $A$ and $B$ be left $\kappa^+$-precompact, and fix $U\in\N$.
Choose a symmetric open neighbourhood $V$ with $V^2\subseteq U$ and a
cover $B\subseteq EV$, where $|E|\leq\kappa$. For every $x\in E$, choose
$F_x\subseteq G$ such that
\[
 |F_x|\leq\kappa\qquad\text{and}\qquad
 A\subseteq F_x(xVx^{-1}).
\]
Then
\[
 AB\subseteq\bigcup_{x\in E}F_xxV^2
       \subseteq\left(\bigcup_{x\in E}F_xx\right)U.
\]
The set of centres has cardinal at most $\kappa$. Induction gives the
assertion for finite products. Countable unions are handled by taking
the union of the corresponding sets of centres.

Put $K=\langle P\rangle$. Since $P$ is symmetric and contains $e$,
\[
 K=\bigcup_{n\geq 1}P^n.
\]
Hence $K$ is left $\kappa^+$-precompact in $G$. The subgroup $K$ is open.
For a target neighbourhood contained in $K$, every covering translate
that meets $K$ has its centre in $K$. Therefore $K$ is $\kappa$-narrow.
\end{proof}

For a singular cardinal $\beta$ of countable cofinality, a countable
union of left $\beta$-precompact sets need not be left
$\beta$-precompact. We show that a subgroup generated by a symmetric
left $\beta$-precompact neighbourhood still has this property when
$\beta>\aleph_0$. The proof uses one cover of $P^2$ by translates of
$P$.

\begin{lemma}\label{lem:product}
Let $\beta$ be an infinite cardinal. In an FSIN group, finite products
of left $\beta$-precompact sets are left $\beta$-precompact.
\end{lemma}

\begin{proof}
Let $A$ and $B$ be left $\beta$-precompact, and fix $U\in\N$.
Choose a symmetric open neighbourhood $V$ with $V^2\subseteq U$.
By left neutrality, there is a symmetric open neighbourhood $W$ such
that $WB\subseteq BV$. Choose covers
\[
 A\subseteq FW,\qquad B\subseteq EV,
 \qquad |F|,|E|<\beta.
\]
Then
\[
 AB\subseteq FWB\subseteq FBV\subseteq FEV^2\subseteq FEU.
\]
The product of two cardinals less than an infinite cardinal $\beta$
is again less than $\beta$. Thus $AB$ is left $\beta$-precompact, and
induction completes the proof.
\end{proof}

\begin{lemma}\label{lem:generation}
Let $\beta>\aleph_0$ be a cardinal, and let $P=P^{-1}\in\N$ in an
FSIN group. If $P$ is left $\beta$-precompact, then $\langle P\rangle$
is left $\beta$-precompact.
\end{lemma}

\begin{proof}
Put $K=\langle P\rangle$. By Lemma~\ref{lem:product}, $P^2$ is left
$\beta$-precompact. Applying this at the target $P$, choose
\[
 P^2\subseteq FP,\qquad |F|<\beta.
\]
We discard all $f\in F$ for which $fP\cap P^2=\varnothing$. For every
remaining $f$, there is $p\in P$ with $fp\in P^2$. Thus
$f\in P^2P^{-1}\subseteq K$.

We have $P^n\subseteq F^{n-1}P$ for all $n\geq 1$, where
$F^0=\{e\}$. Indeed, the assertion holds for $n=1$, and
\[
 P^{n+1}\subseteq F^{n-1}P^2\subseteq F^nP
\]
gives the induction step. Put
\[
 F^*=\bigcup_{n<\omega}F^n.
\]
Since $P$ is symmetric and contains $e$, it follows that
\[
 K=\bigcup_{n\geq 1}P^n\subseteq F^*P,
 \qquad |F^*|\leq\max\{\aleph_0,|F|\}<\beta.
\]
Now fix $U\in\N$. Choose $D_U\subseteq G$ with
$|D_U|<\beta$ and $P\subseteq D_UU$. Then
\[
 K\subseteq F^*D_UU,\qquad |F^*D_U|<\beta.
\]
This proves that $K$ is left $\beta$-precompact.
\end{proof}

The set $F^*$ in this proof is fixed before $U$ is chosen. The last
estimate therefore multiplies only two cardinals less than $\beta$;
it is valid even when $\cf(\beta)=\omega$. The assumption
$\beta>\aleph_0$ cannot be omitted: the precompact neighbourhood
$(-1,1)$ in $(\mathbb R,+)$ generates the nonprecompact group
$\mathbb R$.

\section{Main results}\label{sec:permutations}

The use of uniformly discrete subsets to test the SIN property goes
back to Itzkowitz; see Hern\'andez~\cite[Lemma~2]{Hernandez00}.
That criterion characterizes SIN groups by uniform control of conjugation
on each left uniformly discrete subset. Here we obtain the required
control through local permutations on separated sets.

We first choose a continuous pseudometric at a prescribed neighbourhood.
This follows from the continuous prenorm construction
in~\cite[Theorem~3.3.9]{AT08}. We include a proof to specify the separation
estimates used below.
\begin{lemma}\label{lem:metric}
For every $A\in\N$, there are a continuous left-invariant pseudometric
$d$ on $G$ and $\varepsilon>0$ such that
$\{g:d(e,g)<\varepsilon\}\subseteq A$.
\end{lemma}

\begin{proof}
Put $U_0=G$. Choose a symmetric open $U_1\in\N$ contained in $A$, and
recursively choose symmetric open $U_{n+1}\in\N$ with
$U_{n+1}^3\subseteq U_n$.
An edge from $x$ to $y$ labelled $n$ has cost $2^{-n}$ and is allowed
when $x^{-1}y\in U_n$. Define $d(x,y)$ to be the infimum of the costs
of finite labelled chains from $x$ to $y$. Every pair has an edge
labelled $0$, so $d$ is finite. Reversing and concatenating chains shows
that $d$ is a pseudometric. Left translation of chains shows that it
is left invariant.

We claim that a chain of cost less than $2^{-n}$ has its total increment
in $U_n$. We prove this simultaneously for all $n\ge0$ by induction
on the number of edges. For a nonempty chain of total cost $s<2^{-n}$,
choose an edge containing the halfway point of its ordered sum of costs.
The chains before and after that edge each have cost at most
$s/2<2^{-(n+1)}$. Their increments belong to $U_{n+1}$ by induction.
The middle edge has cost less than $2^{-n}$, so its label is at least
$n+1$ and its increment also belongs to $U_{n+1}$.
The total increment therefore belongs to $U_{n+1}^3\subseteq U_n$.
An empty side chain has increment $e$, so it satisfies the same assertion.

It follows that $d(e,g)<1/2$ implies $g\in U_1\subseteq A$.
Conversely, $g\in U_n$ implies $d(e,g)\le2^{-n}$.
Thus $d(e,g)\to0$ as $g\to e$. The inequality
\[
 |d(x,y)-d(x',y')|\le d(x,x')+d(y,y')
\]
and left invariance now give joint continuity of $d$.
\end{proof}

\begin{definition}
Let $G$ be a group and $U\in \N$. A subet $B\subseteq G$ is called {\it $U$-disjoint} if $b_1U\cap b_2U=\emptyset $ for each $b_1\neq b_2$ in $B$.
\end{definition}

\begin{lemma}\label{Lepermutation}
Let $G$ be topological group and $U,V\in \N$ with $U,V$ being symmetric and $A\subseteq G$ a $V$-disjoint set such that $UA\subseteq AV$. For each $u\in U$ defined $\phi_u:A\rightarrow A$ by $\phi_u(a)=a'$ with $ua=a'v$, where $a'$ is the unique element of $A$ satisfying $ua=a'v$. Then $\phi_u$ is a permutation of $A$ such that $\phi_u^{-1}=\phi_{u^{-1}}$. Denote by $\Gamma(U,V)$ a subgroup of the permutation group $\Sym(A)$ generated by the set $\{\phi_u:u\in U\}$. Then $\Gamma(U',V')\subseteq \Gamma(U,V)$ whenever $U'\subseteq U$, $V'\subseteq V$ and $U'A\subseteq AV'$.
\end{lemma}

\begin{proof}
We first show that $\phi_u$ is well-defined. For fixed $u\in U$ and $a\in A$, the condition $UA\subseteq AV$ ensures that $ua\in AV$, so there exist $a'\in A$ and $v\in V$ with $ua=a'v$. Suppose there are two such pairs: $ua=a_1v_1=a_2v_2$ with $a_i\in A$, $v_i\in V$. Then $a_1v_1=a_2v_2$, hence $a_1V\cap a_2V\neq\varnothing$. Since $A$ is $V$-disjoint, we must have $a_1=a_2$. Thus $a'$ is uniquely determined, so $\phi_u(a)$ is well-defined.

Next we prove that $\phi_u$ is bijective. We will construct its inverse explicitly. Since $U$ is symmetric, $u^{-1}\in U$. Applying the same well-definedness to $u^{-1}$, for every $b\in A$ there exist $a\in A$ and $w\in V$ such that
\[
u^{-1}b = a w.
\]
Multiplying both sides on the left by $u$ gives
\[
b = u a w.
\]
Rearranging, $u a = b w^{-1}$. Because $V$ is symmetric, $w^{-1}\in V$. By the definition of $\phi_u$, the relation $u a = b w^{-1}$ with $w^{-1}\in V$ implies $\phi_u(a)=b$. Hence for every $b\in A$ there exists $a\in A$ with $\phi_u(a)=b$; so $\phi_u$ is surjective.

Now take any $a\in A$ and put $a'=\phi_u(a)$. Then $ua=a'v$ for some $v\in V$. We have
\[
u^{-1}a' = a v^{-1}.
\]
Since $v^{-1}\in V$, the definition of $\phi_{u^{-1}}$ applied to the element $a'$ yields
\[
\phi_{u^{-1}}(a') = a.
\]
Thus $\phi_{u^{-1}}(\phi_u(a))=a$ for every $a\in A$, so $\phi_{u^{-1}}\circ \phi_u = \mathrm{id}_A$. Replacing $u$ by $u^{-1}$ in the above argument gives $\phi_u\circ \phi_{u^{-1}} = \mathrm{id}_A$. Therefore $\phi_u$ is a bijection and
\[
\phi_u^{-1} = \phi_{u^{-1}}.
\]
This completes the proof.

\end{proof}

{\bf Marker Environment $\mathcal{M}(G,O)$:} Let $G$ be a Hausdorff FSIN group, fix $O\in\N$.

{\bf 1. Outer scale layer:}\label{Mer1}
 Choose a symmetric open $S\in\N$ with $S^3\subseteq O$.
By Lemma~\ref{lem:metric}, choose $d$ and $r>0$ such that, on putting
$B_t=\{g:d(e,g)<t\}$, we have
\begin{equation}\label{eq:scale}
 B_{4r}\subseteq S,\qquad S^{-1}B_rS\subseteq O.
\end{equation}
Choose $C\subseteq G$ maximal subject to
$d(c,c')\ge4r$ for distinct $c,c'\in C$.
Maximality gives
\begin{equation}\label{eq:net}
 G=CB_{4r}.
\end{equation}
This layer defines the objects:
\[
S,\quad d,\quad r,\quad C.
\]

{\bf 2. Packing layer:}
Choose arbitrary symmetric open neighbourhoods $P\in\N$ and $Q\in\N$, and a nonempty set $D\subseteq P$ satisfying the packing condition:
\begin{equation}\label{eq:packing}
 D \text{~is~} Q\text{-disjoint and contained in P},
 \qquad PQ^2P^{-1}\subseteq B_r,
\end{equation}
where $B_r$ is from Outer scale layer. We call {\it $D$ a packing at $(Q,P)$}.

{\bf 3. Inner neighbourhood $L$:} Choose a symmetric $R_D\in\N$
contained in $Q$. By Theorem \ref{thm:neutral}, choose a symmetric
$L\in\N$ contained in $P$ such that $LD\subseteq DR_D$. We call {\it Inner neighbourhood $L$ induced by $(D,Q,P)$}
.
By Lemma \ref{Lepermutation}, each $h\in L$ determines a permutation $\eta_h$ on $D$ such that
\begin{equation}\label{eq:inner}
 ha\in\eta_h(a)R_D\qquad(h\in L,\ a\in D).
\end{equation}

{\bf 4. Outer source $W$ and outer group $\Gamma$:} For the symmetric open $L$ from the previous layer.  By Theorem \ref{thm:neutral}, the left neutrality of
$C$ gives a symmetric open $W\in\N$ such that $WC\subseteq CL$.
By Lemma \ref{Lepermutation}, each $u\in W$ determines a permutation $\varphi_u$ on $C$ such that
\begin{equation}\label{eq:outer}
 uc\in\varphi_u(c)L.
\end{equation}
By Lemma \ref{Lepermutation}, we may put
\begin{equation}\label{eq:outergroup}
 \Gamma(W,L)=\langle\varphi_u:u\in W\rangle\le\Sym(C).
\end{equation}
such that
\begin{equation}\label{eq:subgroup}
 \Gamma(W',L')\le \Gamma(W,L),
\end{equation}
 whenever $W'\subseteq W$, $L'\subseteq L$ and $W'D\subseteq DL'$.

\begin{lemma}\label{lem:markers}
Retain a marker environment $\mathcal{M}(G,O)$ as above. Let $\tau:C\to I$ be a colouring and let $I_*\subseteq I$.
Suppose there are nonempty subsets $A_i\subseteq D$, $i\in I_*$, such that
\begin{equation}\label{eq:independent}
 A_j\ne\eta_h(A_i)\qquad(i\ne j,\ i,j\in I_*,\ h\in L).
\end{equation}
There is a symmetric open $E\in\N$ contained in $W$ such that
\[
 \tau(\varphi_u(c))=\tau(c)
 \qquad(u\in E,\ \tau(c)\in I_*).
\]
If $I_*=I$ and $\tau$ distinguishes $\Gamma(W,L)$, then $O$ contains
a conjugation-invariant open identity neighbourhood.
\end{lemma}

\begin{proof} The following $P,Q,R_D$ are from Machine \ref{Mer1} for $O$.

Put
\[
 M=\bigcup_{\tau(c)\in I_*}cA_{\tau(c)}.
\]
The sets $caQ$, $(c,a)\in C\times D$, are pairwise disjoint.
Indeed, $caq=c'a'q'$ gives
\[
 c^{-1}c'=aqq'^{-1}a'^{-1}\in PQ^2P^{-1}\subseteq B_r.
\]
Separation first gives $c=c'$, and disjointness of the sets $aQ$
then gives $a=a'$.

By left neutrality of $M$, choose a symmetric open $E\in\N$ contained
in $W$ such that $EM\subseteq MQ$.
Fix $u\in E$ and $c\in C$ with $\tau(c)\in I_*$, and write
$uc=c'h$ with $c'=\varphi_u(c)$ and $h\in L$.
For $a\in A_{\tau(c)}$, formula~\eqref{eq:inner} gives
\[
 u(ca)=c'ha\in c'\eta_h(a)R_D\subseteq c'\eta_h(a)Q.
\]
The inclusion $EM\subseteq MQ$ and disjointness of all sets $caQ$
show that $\tau(c')\in I_*$ and
\[
 \eta_h(A_{\tau(c)})\subseteq A_{\tau(c')}.
\]
Here nonemptiness of $A_{\tau(c)}$ ensures that the target colour belongs
to $I_*$. Since $u^{-1}c'=ch^{-1}$, the same comparison for $u^{-1}$
gives $\eta_{h^{-1}}(A_{\tau(c')})\subseteq A_{\tau(c)}$ by Lemma \ref{Lepermutation}.
Thus the preceding inclusion is an equality, and
\eqref{eq:independent} implies $\tau(c')=\tau(c)$.

If $I_*=I$ and $\tau$ is distinguishing, every $\varphi_u$, $u\in E$,
is the identity. Hence $c^{-1}Ec\subseteq L\subseteq B_r$ for all $c\in C$.
For $g=cs$ as in~\eqref{eq:net}, we have $s\in B_{4r}\subseteq S$, so
\[
 g^{-1}Eg\subseteq S^{-1}B_rS\subseteq O.
\]
Therefore $\bigcup_{g\in G}g^{-1}Eg$ is the required open invariant
identity neighbourhood.
\end{proof}


\begin{lemma}\label{cor:gain}
Let an infinite set $D$ with $|D|=\mu$ be a packing at $(Q,P)$, that is, $D,P,Q$ satisfy \eqref{eq:packing} in Marker Environment $\mathcal{M}(G,O)$. Then for every $H\in\N$, there is a set $D'$ packing at $(E,P)$, an iner neighbourhood $L'\subseteq P\cap H$ induced by $(D',E,P)$ and a family $\{A_i:i\in I\}$ of nonempty subsets of the
set $D'$ such that
\[
  |I|>\mu,
  \qquad
  A_j\neq\eta'_h(A_i)\quad(i\neq j,\ h\in L').
\]
\end{lemma}

\begin{proof}
Put $\lambda=2^\mu$. Choose a symmetric open neighbourhood $V_D$
such that $V_D^2\subseteq Q$. By left neutrality of $D$, there is a
symmetric open neighbourhood $L\in \N$ with $L\subseteq H\cap P$ such that
$LD\subseteq DV_D$. Then $L$ is an iner neighbourhood induced by $(D,Q,P)$. By Lemma \ref{Lepermutation}, the local permutations $\eta_h\in\Sym(D)$ can be
defined by
\begin{equation}
  hd\in\eta_h(d)V_D\qquad(h\in L,\ d\in D).
  \label{eq:old-permutations}
\end{equation}
Put
\[\Sigma=\{\eta_h:h\in L\}
\] and, for nonempty set $B\subseteq D$, put
\[\Sigma B=\{\eta_h(B):h\in L\}.\]
 We have
$\Sigma^{-1}=\Sigma$ and $\eta_e=\mathrm{id}_D$. Moreover,
$|\Sigma B|\leq|\Sym(D)|=\lambda$ for every $B\subseteq D$.

{\bf Case 1:} Suppose that $|\Sigma B|<\lambda$ for every nonempty $B\subseteq D$.

We choose nonempty set $A_\alpha\subseteq D$ recursively for $\alpha<\cf(\lambda)$.
At stage $\alpha$, the set
\[
  \bigcup_{\gamma<\alpha}\Sigma A_\gamma
\]
has cardinality less than $\lambda$. Indeed, it is the union of fewer
than $\cf(\lambda)$ sets, each of cardinality less than $\lambda$.
We can therefore choose a nonempty set $A_\alpha\subseteq D$ outside this
union. The required family
with $D'=D$, $L'=L$, $E=Q$ and $I=\cf(\lambda)$. Since $\Sigma^{-1}=\Sigma$,
the resulting family satisfies the
required separation condition in both directions. The inequality $\cf(2^\mu)>\mu$ is a consequence of K\"onig's theorem; see
\cite[Corollary~5.12]{Jech03}.

{\bf Case 2:} Suppose that $|\Sigma B|=\lambda$ for some nonempty $B\subseteq D$.

By left neutrality of $B$, choose a symmetric open neighbourhood
$E\subseteq L\cap Q$ such that $EB\subseteq BV_D$.
For $v\in E$ and $b\in B$, comparison with
\eqref{eq:old-permutations} inside the pairwise disjoint sets $dQ$
gives $\eta_v(b)\in B$. Thus $\eta_v(B)\subseteq B$.
Applying the same argument to $v^{-1}$ and using
$\eta_{v^{-1}}=\eta_v^{-1}$, we obtain
\begin{equation}
  \eta_v(B)=B\qquad(v\in E).
  \label{eq:subset-stabilized}
\end{equation}

For each $Y\in\Sigma B$, choose $t_Y\in L$ with
$\eta_{t_Y}(B)=Y$, and put $T=\{t_Y:Y\in\Sigma B\}$.
Distinct sets $Y$ have distinct representatives, so $|T|=\lambda$.
We show that the sets $t_YE$ are pairwise disjoint. Suppose that
$t_Yv=t_Zw$, where $v,w\in E$. For each $b\in B$, two applications
of \eqref{eq:old-permutations} give
\[
  t_Yvb\in\eta_{t_Y}\eta_v(b)V_D^2,
  \qquad
  t_Zwb\in\eta_{t_Z}\eta_w(b)V_D^2.
\]
The two left-hand sides are equal. Since $V_D^2\subseteq Q$ and the
sets $dQ$ are pairwise disjoint, their centres are equal. Taking
the images of $B$ and using \eqref{eq:subset-stabilized}, we obtain
$Y=Z$. Also, $T\subseteq L\subseteq P$ and
\[
  PE^2P^{-1}\subseteq PQ^2P^{-1}\subseteq B_r.
\]

By left neutrality of $T$, choose a symmetric open neighbourhood
$J\subseteq E\cap H_0$ such that $JT\subseteq TE$. By Lemma \ref{Lepermutation}, the local permutations $\xi_j$ on $T$ can be
defined by
\begin{equation}
  jt\in\xi_j(t)E\qquad(j\in J,\ t\in T).
  \label{eq:old-permutations}
\end{equation}

We next compare these permutations with the permutations of $D$.
Suppose that $\xi_j(t_Y)=t_Z$. There is $v\in E$ such that
$jt_Y=t_Zv$. For $b\in B$, equation
\eqref{eq:old-permutations} gives
\[
  jt_Yb\in\eta_j\eta_{t_Y}(b)V_D^2,
  \qquad
  t_Zvb\in\eta_{t_Z}\eta_v(b)V_D^2.
\]
All permutations in these expressions are defined, since
$j,v,t_Y,t_Z\in L$. Equality of the left-hand sides and disjointness
of the sets $dQ$ imply
\[
  \eta_j\eta_{t_Y}(b)=\eta_{t_Z}\eta_v(b).
\]
Taking the images of $B$, we conclude that
\begin{equation}
  \eta_j(Y)=Z.
  \label{eq:permutation-comparison}
\end{equation}

Let $\theta:T\longrightarrow\Sigma B$ be the bijection defined by
$\theta(t_Y)=Y$. Equation \eqref{eq:permutation-comparison} shows
that $\theta\xi_j\theta^{-1}$ is the restriction to $\Sigma B$ of
the action of $\eta_j$ on $\Pow(D)$. Since $J$ is symmetric, these
restrictions and their inverses preserve $\Sigma B$. Put
$\Xi=\langle\xi_j:j\in J\rangle$. Therefore
$\Xi$ is isomorphic to the image of $\langle\eta_j:j\in J\rangle$
under its action on $\Sigma B$, and
\[
  |\Xi|\leq|\Sym(D)|=\lambda.
\]
Every $\Xi$-orbit on $\Pow(T)$ consequently has cardinality at most
$\lambda$. There are $2^\lambda$ such orbits: if their number were
$p<2^\lambda$, then
\[
  |\Pow(T)|\leq p\cdot\lambda
  \leq\max\{p,\lambda\}<2^\lambda,
\]
a contradiction. Removing the orbit of the empty set does not
change this cardinality.

Choose one representative from each nonempty $\Xi$-orbit on
$\Pow(T)$. These representatives are separated by every $\xi_j$,
and there are $2^{2^\mu}>\mu$ of them. The required family with $D'=T$, $L'=J$, $E$ and $I=2^{2^\mu}$.
proves the assertion.
\end{proof}

We retain the pseudometric $d$, the number $r$, and the set $C$ from
Outer scale layer in Marker Environment $\mathcal{M}(G,O)$. If an open subgroup $K$ is $\kappa$-narrow, then
Lemma~\ref{lem:permanence} shows that each double coset $KxK$ is left
$\kappa^+$-precompact. Cover it by at most $\kappa$ left translates of
$B_r$. Each translate has $d$-diameter less than $2r$ and therefore
contains at most one point of $C$. Consequently,
\begin{equation}\label{eq:doublecoset}
 |C\cap KxK|\leq\kappa\qquad(x\in G).
\end{equation}

The next lemma fixes an open subgroup and compares its covering
cardinal with the sizes of packings in one neighbourhood.

\begin{lemma}\label{lem:scale}
Retain the set $B_r$ and the set $C$ from
Outer scale layer in Marker Environment $\mathcal{M}(G,O)$. Choose symmetric open neighbourhoods $P$ and $R$ such that
\[
 PRP^{-1}\subseteq B_r,
\]
and put $K=\langle P\rangle$ and $\beta=\ib(K)$.
Fix symmetric open neighbourhoods $L_0$ and $W_0$ satisfying
\[
 L_0\subseteq P\cap B_r,\qquad W_0\subseteq P,
 \qquad W_0C\subseteq CL_0,
\]
and put $\Gamma_0=\Gamma(W_0,L_0)$ and $d_0=\mathrm D(\Gamma_0,C)$.
Then $d_0\leq\beta$.

For every infinite cardinal $\kappa<\beta$, there is a symmetric open
neighbourhood $Q_\kappa$ and a maximal packing $D_\kappa\subseteq P$ at
$Q_\kappa$ such that $|D_\kappa|>\kappa$. The neighbourhoods defining
the inner and outer permutations can be chosen with
$L_\kappa\subseteq L_0$ and $W_\kappa\subseteq W_0$. The resulting group
$\Gamma_\kappa=\Gamma(W_\kappa,L_\kappa)$ is a subgroup of $\Gamma_0$.
\end{lemma}

\begin{proof}
For $u\in W_0$ and $c\in C$, write
\[
 uc=\varphi_u(c)\ell\qquad(\ell\in L_0).
\]
Since $u,\ell\in P\subseteq K$, we have
$\varphi_u(c)=uc\ell^{-1}\in KcK$. Thus every $\Gamma_0$-orbit through
$c$ lies in $C\cap KcK$. By \eqref{eq:doublecoset}, its cardinal is at
most $\beta$. Colour each orbit injectively, using the same set of
$\beta$ colours for all orbits. A colour-preserving element of
$\Gamma_0$ fixes every point, so $d_0\leq\beta$.

Fix an infinite cardinal $\kappa<\beta$. If $P$ were left
$\kappa^+$-precompact, Lemma~\ref{lem:permanence} would make $K$
$\kappa$-narrow. Hence there is $U_\kappa\in\N$ such that $P$ cannot
be covered by at most $\kappa$ left translates of $U_\kappa$.
Choose a symmetric open neighbourhood $Q_\kappa$ with
\[
 Q_\kappa^2\subseteq R\cap U_\kappa,
\]
and a maximal set $D_\kappa\subseteq P$ for which the translates
$dQ_\kappa$, $d\in D_\kappa$, are pairwise disjoint. Maximality gives
\[
 P\subseteq D_\kappa Q_\kappa^2
  \subseteq D_\kappa U_\kappa,
\]
so $|D_\kappa|>\kappa$. Moreover,
\[
 PQ_\kappa^2P^{-1}\subseteq PRP^{-1}\subseteq B_r.
\]
Thus $D_\kappa$ is a packing at $(Q_\kappa,P)$ in the sense of
\eqref{eq:packing}.

Choose a symmetric open neighbourhood $V_\kappa$ with
$V_\kappa^2\subseteq Q_\kappa$. By left neutrality, choose a symmetric
open neighbourhood $L_\kappa\subseteq L_0$ such that
\[
 L_\kappa D_\kappa\subseteq D_\kappa V_\kappa.
\]
Apply left neutrality to $C$ and choose a symmetric open neighbourhood
$W_\kappa\subseteq W_0$ such that
$W_\kappa C\subseteq CL_\kappa$.
Since $L_\kappa\subseteq L_0$, uniqueness of the centre in $C$ shows
that every new outer permutation agrees with the corresponding
permutation used to define $\Gamma_0$. Hence
$\Gamma_\kappa\leq\Gamma_0$ \eqref{eq:subgroup}.
\end{proof}

\begin{lemma}\label{lem:nonattainment}
Retain the data of Lemma~\ref{lem:scale}. Suppose that
$\beta>\aleph_0$ and that every infinite packing contained in the
fixed neighbourhood $P$ has cardinal less than $\beta$. Then
$d_0<\beta$.
\end{lemma}

\begin{proof}
Fix $U\in\N$. Choose a symmetric open neighbourhood $Q$ with
$Q^2\subseteq R\cap U$, and a maximal set $D\subseteq P$ for which
the translates $dQ$, $d\in D$, are pairwise disjoint. Since
\[
 PQ^2P^{-1}\subseteq PRP^{-1}\subseteq B_r,
\]
the set $D$ is a packing at $Q$. If $D$ is infinite, the hypothesis
gives $|D|<\beta$; the same inequality holds when $D$ is finite.
Maximality yields
\[
 P\subseteq DQ^2\subseteq DU.
\]
Thus $P$ is left $\beta$-precompact. By
Lemma~\ref{lem:generation}, so is $K=\langle P\rangle$.

The subgroup $K$ is open. By Theorem~\ref{thm:core}, its normal core
$N=\core_G(K)$ is open and normal in $G$. A cover of $K$ by fewer
than $\beta$ left translates of $N$ gives
\[
 \lambda=[K:N]<\beta.
\]
Choose a symmetric open neighbourhood $U_0\subseteq B_r$ and a cover
\[
 K\subseteq F_0U_0,\qquad \rho=|F_0|<\beta.
\]
For every $x\in G$,
$xN\subseteq xK\subseteq xF_0U_0$. Each left translate of $U_0$
has $d$-diameter less than $2r$ and hence meets $C$ in at most one
point. It follows that
\[
 |C\cap xN|\leq\rho\qquad(x\in G).
\]

Let $T$ be a transversal for $K/N$. Since $N$ is normal in the
ambient group $G$, for every $c\in C$ we have
\[
 KcK=\bigcup_{a,b\in T}acbN.
\]
Every $\Gamma_0$-orbit through $c$ lies in $C\cap KcK$, as shown in
the proof of Lemma~\ref{lem:scale}. Therefore
\[
 |\Gamma_0c|\leq\lambda^2\rho.
\]
Colour each orbit injectively, using one common set of
$\max\{\aleph_0,\lambda,\rho\}<\beta$ colours. This colouring is
distinguishing, so $d_0<\beta$.
\end{proof}

\begin{theorem}\label{thm:main}
Every Hausdorff FSIN topological group is SIN.
\end{theorem}

\begin{proof}
Let $G$ be a Hausdorff FSIN group, and fix $O\in\N$.
If $G$ is locally precompact, Theorem~\ref{thm:local} applies.
Suppose that $G$ is not locally precompact.
For $O$, choose $S,d,r,C$ as in Outer scale layer of Marker Environment, and fix
$P,R,K,\beta,L_0,W_0,\Gamma_0,d_0$ as in Lemma~\ref{lem:scale}.
In particular,
\[
 K=\langle P\rangle,\qquad \beta=\ib(K),\qquad
 d_0=\mathrm D(\Gamma_0,C)\le\beta.
\]

{\bf Claim 1:} There is an infinite packing in the sense of \eqref{eq:packing} contained in $P$.

Indeed, $P$ is not precompact, so some $U\in\N$ admits no finite
cover $P\subseteq FU$. Choose a symmetric open $Q\in\N$ with
$Q^2\subseteq R\cap U$, and a maximal $D\subseteq P$ such that $D$ is $Q$-disjoint, that is,
the sets $aQ$, $a\in D$, are pairwise disjoint.
Then $P\subseteq DQ^2\subseteq DU$, so $D$ is infinite.
Also $PQ^2P^{-1}\subseteq PRP^{-1}\subseteq B_r$.
Taking a countably infinite subset of $D$ therefore gives a countably
infinite packing in the same $P$.

{\bf Claim 2:} $P$ contains an infinite packing in the sense of \eqref{eq:packing} of cardinality at least
$d_0$.

If there is one of cardinality at least $\beta$, this follows
from $d_0\le\beta$.
Otherwise every infinite packing in $P$ has cardinality less than
$\beta$. Claim 1 implies
$\beta>\aleph_0$.
By Lemma~\ref{lem:nonattainment}, $d_0<\beta$.
Put $\kappa=\max\{\aleph_0,d_0\}<\beta$.
Lemma~\ref{lem:scale} gives a packing in $P$ of cardinality
$\mu>\kappa$, as required.

Choose such a packing, of infinite cardinality $\mu\ge d_0$.
Apply Lemma~\ref{cor:gain}, for $L_0$, there is a set $D'$ packing at $(E,P)$, an iner neighbourhood $L\subseteq P\cap L_0$ induced by $(D',E,P)$ and a family $\{A_i:i\in I'\}$ of nonempty subsets of the
set $D'$ such that
\[
  |I'|>\mu,
  \qquad
  A_j\neq\eta'_h(A_i)\quad(i\neq j,\ h\in L).
\]

Choose a distinguishing colouring $\tau:C\to I$ for $\Gamma_0$ with
$|I|=d_0$, and assign to its colours distinct members of this family.
By left neutrality of $C$, choose a symmetric open $W\in\N$
contained in $W_0$ such that $WC\subseteq CL$.
Since $L\subseteq L_0$, formula~\eqref{eq:subgroup} gives
\[
 \Gamma(W,L)\le\Gamma_0.
\]
The same colouring therefore distinguishes $\Gamma(W,L)$.
Lemma~\ref{lem:markers}, with $I_*=I$, now gives an open
conjugation-invariant identity neighbourhood contained in $O$.
Since $O$ was arbitrary, $G$ is SIN.
\end{proof}

\begin{corollary}\label{cor:classes}
For Hausdorff topological groups,
\[
 [\mathrm{SIN}]=[\mathrm{SFSIN}]=[\mathrm{FSIN}].
\]
\end{corollary}

\begin{proof}
The definitions give
$[\mathrm{SIN}]\subseteq[\mathrm{SFSIN}]\subseteq[\mathrm{FSIN}]$.
Theorem~\ref{thm:main} gives the reverse inclusion.
\end{proof}

Thus both versions of the Itzkowitz problem discussed
in~\cite{BT07} have affirmative answers.

\end{document}